\documentclass[11pt]{article}
\usepackage[margin=1in]{geometry} 
\usepackage{amssymb,amsfonts,amsmath,bbm,mathrsfs,stmaryrd}
\usepackage{xcolor}
\usepackage{url}

\usepackage{enumerate}

\usepackage{hyperref}
\hypersetup{colorlinks,
             linkcolor=black!75!red,
             citecolor=blue,
             pdftitle={},
             pdfproducer={pdfLaTeX},
             pdfpagemode=None,
             bookmarksopen=true
             bookmarksnumbered=true}

\usepackage{tikz}
\usetikzlibrary{arrows,calc,decorations.pathreplacing,decorations.markings,intersections,shapes.geometric,through,fit,shapes.symbols,positioning,decorations.pathmorphing}

\usepackage{braket}

\usepackage[amsmath,thmmarks,hyperref]{ntheorem}
\usepackage{cleveref}

\creflabelformat{enumi}{#2(#1)#3}

\crefname{section}{Section}{Sections}
\crefformat{section}{#2Section~#1#3} 
\Crefformat{section}{#2Section~#1#3} 

\crefname{subsection}{\S}{\S\S}
\AtBeginDocument{%
  \crefformat{subsection}{#2\S#1#3}%
  \Crefformat{subsection}{#2\S#1#3}%
}

\crefname{subsubsection}{\S}{\S\S}
\AtBeginDocument{%
  \crefformat{subsubsection}{#2\S#1#3}%
  \Crefformat{subsubsection}{#2\S#1#3}%
}

\theoremstyle{plain}

\newtheorem{lemma}{Lemma}[section]
\newtheorem{proposition}[lemma]{Proposition}
\newtheorem{corollary}[lemma]{Corollary}
\newtheorem{theorem}[lemma]{Theorem}

\newtheorem{question}[lemma]{Question}

\theoremstyle{nonumberplain}

\theoremstyle{plain}
\theorembodyfont{\upshape}
\theoremsymbol{\ensuremath{\blacklozenge}}

\newtheorem{definition}[lemma]{Definition}
\newtheorem{example}[lemma]{Example}
\newtheorem{remark}[lemma]{Remark}
\newtheorem{convention}[lemma]{Convention}

\crefname{definition}{definition}{definitions}
\crefformat{definition}{#2definition~#1#3} 
\Crefformat{definition}{#2Definition~#1#3} 

\crefname{ex}{example}{examples}
\crefformat{example}{#2example~#1#3} 
\Crefformat{example}{#2Example~#1#3} 

\crefname{remark}{remark}{remarks}
\crefformat{remark}{#2remark~#1#3} 
\Crefformat{remark}{#2Remark~#1#3} 

\crefname{convention}{convention}{conventions}
\crefformat{convention}{#2convention~#1#3} 
\Crefformat{convention}{#2Convention~#1#3} 

\crefname{notation}{notation}{notations}
\crefformat{notation}{#2notation~#1#3} 
\Crefformat{notation}{#2Notation~#1#3} 

\crefname{table}{table}{tables}
\crefformat{table}{#2table~#1#3} 
\Crefformat{table}{#2Table~#1#3}

\crefname{lemma}{lemma}{lemmas}
\crefformat{lemma}{#2lemma~#1#3} 
\Crefformat{lemma}{#2Lemma~#1#3} 

\crefname{proposition}{proposition}{propositions}
\crefformat{proposition}{#2proposition~#1#3} 
\Crefformat{proposition}{#2Proposition~#1#3} 

\crefname{corollary}{corollary}{corollaries}
\crefformat{corollary}{#2corollary~#1#3} 
\Crefformat{corollary}{#2Corollary~#1#3} 

\crefname{theorem}{theorem}{theorems}
\crefformat{theorem}{#2theorem~#1#3} 
\Crefformat{theorem}{#2Theorem~#1#3} 

\crefname{enumi}{}{}
\crefformat{enumi}{(#2#1#3)}
\Crefformat{enumi}{(#2#1#3)}

\crefname{assumption}{assumption}{Assumptions}
\crefformat{assumption}{#2assumption~#1#3} 
\Crefformat{assumption}{#2Assumption~#1#3} 

\crefname{equation}{}{}
\crefformat{equation}{(#2#1#3)} 
\Crefformat{equation}{(#2#1#3)}

\numberwithin{equation}{section}
\renewcommand{\theequation}{\thesection-\arabic{equation}}

\theoremstyle{nonumberplain}
\theoremsymbol{\ensuremath{\blacksquare}}

\newtheorem{proof}{Proof}
\newcommand\pf[1]{\newtheorem{#1}{Proof of \Cref{#1}}}

\newcommand\bC{{\mathbb C}}

\newcommand\bR{{\mathbb R}}
\newcommand\bS{{\mathbb S}}

\newcommand\bZ{{\mathbb Z}}

\newcommand\cO{{\mathcal O}}

\newcommand\fa{{\mathfrak a}}

\newcommand\fg{{\mathfrak g}}

\newcommand\fk{{\mathfrak k}}

\newcommand\fn{{\mathfrak n}}

\newcommand\fp{{\mathfrak p}}

\newcommand\fz{{\mathfrak z}}

\DeclareMathOperator{\id}{id}

\newcommand\numberthis{\addtocounter{equation}{1}\tag{\theequation}}

\newcommand{\cat}[1]{\textsc{#1}}

\renewcommand{\square}{\mathrel{\Box}}

\title{Chain-center duality for locally compact groups}
\author{Alexandru Chirvasitu}

\begin{document}

\date{}

\newcommand{\Addresses}{{
  \bigskip
  \footnotesize

  \textsc{Department of Mathematics, University at Buffalo, Buffalo,
    NY 14260-2900, USA}\par\nopagebreak \textit{E-mail address}:
  \texttt{achirvas@buffalo.edu}

}}

\maketitle

\begin{abstract}
  The chain group $C(G)$ of a locally compact group $G$ has one generator $g_{\rho}$ for each irreducible unitary $G$-representation $\rho$, a relation $g_{\rho}=g_{\rho'}g_{\rho''}$ whenever $\rho$ is weakly contained in $\rho'\otimes \rho''$, and $g_{\rho^*}=g_{\rho}^{-1}$ for the representation $\rho^*$ contragredient to $\rho$. $G$ satisfies chain-center duality if assigning to each $g_{\rho}$ the central character of $\rho$ is an isomorphism of $C(G)$ onto the dual $\widehat{Z(G)}$ of the center of $G$.

  We prove that $G$ satisfies chain-center duality if it is (a) a compact-by-abelian extension, (b) connected nilpotent, (c) countable discrete icc or (d) connected semisimple; this generalizes M. M\"{u}ger's result compact groups satisfy chain-center duality.
\end{abstract}


\vspace{.5cm}


\tableofcontents

\section*{Introduction}

The motivation for the paper is the following representation-theoretic interpretation of the center $Z(G)$ of a compact group $G$ discovered in \cite{mug}:
\begin{itemize}
\item On the one hand, one can impose a ``universal grading'' on the category $\mathrm{Rep}(G)$ of unitary $G$-representations by assigning each irreducible $\rho\in\widehat{G}$ an abstract generator $g_{\rho}$ and imposing the relation
  \begin{equation*}
    g_{\rho} = g_{\rho'}g_{\rho''}
  \end{equation*}
  whenever we have a direct-summand inclusion
  \begin{equation*}
    \rho\le \rho'\otimes\rho''.
  \end{equation*}
  This is the {\it chain group} $C(G)$ of \cite[Proposition 2.3]{mug}. 
\item On the other, there is an obvious grading of $\mathrm{Rep}(G)$ by the Pontryagin dual $\widehat{Z(G)}$: the center acts by scalars on any irreducible representation $\rho\in\widehat{G}$ by Schur's Lemma \cite[(8.6)]{rob}, so $\rho$ naturally gets assigned a character $\chi_{\rho}\in \widehat{Z(G)}$; this assignment then satisfies the same type of tensor-product-compatibility relation:
  \begin{equation*}
    \rho\le \rho'\otimes\rho''\Rightarrow \chi_{\rho} = \chi_{\rho'}\chi_{\rho''}.
  \end{equation*}
\end{itemize}
The main result of \cite{mug} is that these two procedures are equivalent. Formally, \cite[Theorem 3.1]{mug} says that
\begin{equation}\label{eq:canpre}
  C(G)\ni g_{\rho}\mapsto \chi_{\rho}\in \widehat{Z(G)}\text{ for }\rho\in\widehat{G}
\end{equation}
is an isomorphism between the chain group $C(G)$ and the dual $\widehat{Z(G)}$ of the center: the {\it chain-center duality} of this paper's title. As that same title suggests, the focus here is on {\it locally} compact groups. Extending the framework for considering such duality results is not difficult, though some modifications are needed.

One problem is that in general, for locally compact groups, appearing as a direct summand (or not) is a poor indicator of ``containment''. A much more reasonable notion is that of {\it weak} containment (\cite[Chapter 3, \S 4.5, Definition 3]{kir} or \cite[Theorem 3.4.4]{dixc}), denoted throughout by `$\prec$' or `$\preceq$'. In view of this, our reworking of \cite[Definition 2.1, Proposition 2.3]{mug} reads

\begin{definition}\label{def:chain}
  Let $G$ be a locally compact group. The {\it chain group} $C(G)$ is group defined
  \begin{enumerate}[(a)]
  \item\label{item:1} by one generator $g_\rho$ for each irreducible unitary representation $\rho\in \widehat{G}$;
  \item\label{item:2} subject to $g_\rho=g_{\rho'}g_{\rho''}$ for any weak containment relation
    \begin{equation*}
      \rho\preceq \rho'\otimes \rho'';
    \end{equation*}
  \item\label{item:3} as well as $g_{\rho^*}=g_rho^{-1}$ for the {\it contragredient} representation $\rho^*$ of $\rho$ \cite[Definition A.1.10]{bdv}.
  \end{enumerate}
  The chain group has a natural topology: since
  \begin{itemize}
  \item the unitary dual $\widehat{G}$ surjects onto the set of generators $\{g_\rho\}$;
  \item and the relations ensure that the image of that surjection in fact encompasses all of $C(G)$,
  \end{itemize}
  we can equip the latter with the quotient of the {\it Fell topology} \cite[Definition F.2.1]{bdv} resulting from said surjection $\widehat{G}\to C(G)$.
\end{definition}

Schur's Lemma and the canonical map \Cref{eq:canpre} make just as much sense as for compact groups, so one can define the locally compact group $G$ to be {\it chain-center dual} or {\it cc-dual} for short if \Cref{eq:canpre} is an isomorphism (see \Cref{def:ccdual}). The obvious question flows naturally:

\begin{question}
  Are all locally compact groups cc-dual? 
\end{question}

I do not know the answer, but the results below identify several fairly large classes. Paraphrasing and summarizing:

\begin{theorem}
  Let $G$ be a locally compact group. $G$ is cc-dual in each of the following cases:
  \begin{enumerate}[(a)]
  \item $G$ is a compact-by-abelian extension (i.e. has a closed, abelian normal subgroup with corresponding compact quotient): \Cref{th:cpct-by-ab};
  \item $G$ is connected and nilpotent: \Cref{th:simp-con-nil};
  \item $G$ is countable discrete, and its non-trivial conjugacy classes are infinite (i.e. $G$ is {\it icc}): \Cref{th:icc};
  \item $G$ is a connected semisimple Lie group: \Cref{th:rr1}. 
  \end{enumerate}
\end{theorem}

\subsection*{Acknowledgements}

This work is partially supported by NSF grant DMS-2001128.

\section{Preliminaries}\label{se.prel}

Topological groups are always assumed Hausdorff unless specified otherwise, and `representation' always means unitary (for groups) or $*$-representation on a Hilbert space for a $C^*$-algebra.

For a locally compact group $G$ the symbol $\widehat{G}$ denotes the set of isomorphism classes of irreducible unitary representations (as is standard in the literature: \cite[\S 13.1.4]{dixc}, \cite[Definition 1.46]{kt}, etc.). More generally (though this will not happen often), $\widehat{A}$ denotes the set of (isomorphism classes of) irreducible $*$-representations of the $C^*$-algebra $A$ (cf. \cite[\S 2.3.2]{dixc} or \cite[\S 4.1.1]{ped}).

Some background on {\it type-I} $C^*$-algebras is needed, as covered, say, in \cite[Chapter 9]{dixc} (the term {\it postliminal} \cite[Definition 4.3.1]{dixc} is also in use) or \cite[Chapter 6]{ped}. Locally compact groups are type-I if their universal $C^*$-algebras are (as in \cite[\S 13.9.4]{dixc}, for instance). Numerous equivalent characterizations of the type-I property are known when the $C^*$-algebra in question is separable: see for instance \cite[Theorem 9.1]{dixc}, \cite[Theorem 6.8.7]{ped} or \cite[Theorems IV.1.5.7 and IV.1.5.12]{blk}. Complications ensue in the non-separable case \cite[\S 6.9]{ped}; for that reason, the following convention is in place.

\begin{convention}
  All discussion of type-I algebras / groups will be limited to the separable case (for groups this means $G$ is assumed to be {\it second-countable}, i.e. have a countable basis of open sets \cite[\S 1.3]{mz}).
\end{convention}

We will, at some point, have to consider groups that are also topological spaces, but for which it is unclear whether the multiplication is continuous. We refer to these as {\it groups-with-topology} (cf. \Cref{re:notopgp}) and denote the corresponding category (with continuous group morphisms) by $\mathrm{Gr}_{top}$. By contrast, $\mathrm{TopGp}$ denotes the category of topological groups.

\begin{proposition}\label{pr:topgradj}
  The inclusion functor $\iota:\mathrm{TopGr}\subset \mathrm{Gr}_{top}$ from topological groups to groups-with-topology is a right adjoint.
\end{proposition}
\begin{proof}
  We construct the left adjoint
  \begin{equation*}
    F: \mathrm{Gr}_{top}\to \mathrm{TopGr}.
  \end{equation*}
  On the level of sets and functions between them $F$ is simply the identity; what it does to a group-with-topology is alter that topology, weakening it so as to ensure compatibility with the group operations.

  Concretely, let $G$ be a group-with-topology and denote by $\tau$ that topology. Now consider the weaker topology defined as follows:
  \begin{itemize}
  \item keep only those open sets of $\tau$ whose preimage through the multiplication $G\times G\to G$ is open in the product topology $\tau\times\tau$;
  \item and whose preimage through the inverse $\bullet^{-1}:G\to G$ is also $\tau$-open. 
  \end{itemize}
  Then repeat the procedure with $\tau'$ in place of $\tau$ if necessary, etc. Doing this countably infinitely many times will suffice to produce a topology $\tau'$ weaker than $\tau$ making the original group structure of $G$ into a topological one.

  By the very definition of $\tau'$, any morphism $(G,\tau)\to H$ of groups-with-topology for a topological group $H$ will in fact factor (uniquely) through the canonical continuous map $\id:(G,\tau)\to (G,\tau')$. This universality is what makes $F:(G,\tau)\mapsto (G,\tau')$ the left adjoint to $\iota$.
\end{proof}

\section{Chain groups and centers}\label{se:cc}

\subsection{Generalities}

The chain group $C(G)$ attached to a locally compact group $G$ was defined above (see \Cref{def:chain}). 

\begin{remark}\label{re:notopgp}
  Note that at this stage $C(G)$ has {\it not} been shown to be a topological group: the topology it is equipped with might, in principle, not be compatible with the multiplication.

  While $C(G)$ will be a topological group in the cases of interest (e.g. this is an implicit requirement of \Cref{def:ccdual}), I do not know whether this is so in general. For that reason, we occasionally use the awkward phrase `groups-with-topology' to refer to the category of groups equipped with a perhaps-incompatible topology, with continuous group morphisms.
\end{remark}

The chain group relates to the center $Z(G)$ as follows (cf. \cite[Proposition 2.5]{mug}).

\begin{definition}\label{def:ccdual}
  Since every central element $g\in Z(G)$ acts as a scalar in each irreducible representation $X\in \widehat{G}$ (by Schur's Lemma; e.g. \cite[(8.6)]{rob}), we have a canonical map $\widehat{G}\to \widehat{Z(G)}$ onto the Pontryagin dual of the center of $G$. It is easily checked that that map extends to a continuous surjective (\Cref{le:surj}) morphism.
  \begin{equation}\label{eq:can}
    \cat{can}:C(G)\to \widehat{Z(G)}
  \end{equation}
  $G$ {\it satisfies chain-center duality} or {\it is chain-center or cc-dual} if the map \Cref{eq:can} is an isomorphism of topological groups.
\end{definition}

\begin{example}\label{ex:cpct}
  In the language of \Cref{def:ccdual}, \cite[Theorem 3.1]{mug} says that compact groups are cc-dual (note that in that case the topologies on the groups in question are discrete).  
\end{example}

On the other hand, we also have

\begin{lemma}\label{le:ab}
  Abelian locally compact groups are cc-dual in the sense of \Cref{def:ccdual}.
\end{lemma}
\begin{proof}
  In the abelian case $Z(G)=G$. On the other hand, because irreducible unitary representations are 1-dimensional (e.g. by Schur's lemma again, \cite[(8.6)]{rob}), the dual $\widehat{G}$ is simply the Pontryagin dual group with its standard topology \cite[Example F.2.5]{bdv}. That \Cref{eq:can} is an identification is now clear.
\end{proof}

We make note of the following simple fact, for future reference. 

\begin{lemma}\label{le:surj}
  For an arbitrary locally compact group $G$ the canonical morphism \Cref{eq:can} is onto.
\end{lemma}
\begin{proof}
  Regard an arbitrary character
  \begin{equation*}
    \widehat{Z(G)}\ni \chi:Z(G)\to \bS^1\subset \bC
  \end{equation*}
  as a 1-dimensional unitary representation of the center $Z:=Z(G)$, and consider the induced representation 
  \begin{equation*}
    \rho:=\mathrm{Ind}_Z^G\chi
  \end{equation*}
  \cite[Definition E.1.6]{bdv}. It is easy to see from the definition of the induction procedure that the restriction $\rho|_Z$ back to the center is a sum of copies of $\chi$ (i.e. $Z$ acts in $\rho$ via $\chi$). But then it follows from the Fell-continuity of the restriction operation
  \begin{equation*}
    \mathrm{Rep}(G)\ni \sigma\mapsto \sigma|_Z\in \mathrm{Rep}(Z)
  \end{equation*}
  \cite[Proposition F.3.4]{bdv} that all irreducible representations $\sigma\in \widehat{G}$ weakly contained in $\rho$ (of which there are plenty \cite[Proposition F.2.7]{bdv}) have the same property that $Z$ acts therein via $\chi$. The class of any such $\sigma$ will be mapped by \Cref{eq:can} onto $\chi$, finishing the proof.
\end{proof}

The canonical morphisms \Cref{eq:can} are also functorial, with the appropriate caveats.

\begin{lemma}\label{le:fnct}
  Both $C(G)$ and $\widehat{Z(G)}$ are contravariant functors from the category $\mathcal{LCG}_{dense}$ of locally compact groups with dense-image morphisms to that of abelian groups-with-topology, and the morphisms \Cref{eq:can} constitute a natural transformation.
\end{lemma}
\begin{proof}
  The naturality claim is easy to check once we have functoriality. Consider a dense-image morphism $f:G\to H$, and note first that
  \begin{equation*}
    f(Z(G))\subseteq Z(H)
  \end{equation*}
  (because of the image density) and hence, upon dualizing, $f$ induces a map $\widehat{Z(H)}\to \widehat{Z(G)}$ in the opposite direction.

  As to chain groups, the dense-image condition now ensures that unitary-representation restriction along $f$ preserves irreducibility. The definitions of $C(G)$ and $C(H)$ by generators and relations starting with the irreducibles then makes it clear that we have an induced map $C(H)\to C(G)$.
\end{proof}

\Cref{re:notopgp} notwithstanding, the topologies on $C(G)$ and $\widehat{Z(G)}$ will not play much of a role in assessing cc-duality:

\begin{lemma}\label{le:bijenough}
  A locally compact group $G$ is cc-dual in the sense of \Cref{def:ccdual} if an only if the canonical map \Cref{eq:can} is a bijection.
\end{lemma}
\begin{proof}
  One implication is obvious, so we handle the other. Assume \Cref{eq:can} is bijective. It fits into a commutative diagram
  \begin{equation}\label{eq:triang}
    \begin{tikzpicture}[auto,baseline=(current  bounding  box.center)]
      \path[anchor=base] 
      (0,0) node (l) {$\widehat{G}$}
      +(2,1) node (u) {$C(G)$}
      +(2,-1) node (d) {$\widehat{Z(G)}$}
      ;
      \draw[->] (l) to[bend left=6] node[pos=.5,auto] {$\scriptstyle $} (u);
      \draw[->] (u) to[bend left=6] node[pos=.5,auto] {$\scriptstyle \cat{can}$} (d);
      \draw[->] (l) to[bend right=6] node[pos=.5,auto,swap] {$\scriptstyle $} (d);
    \end{tikzpicture}
  \end{equation}
  where the top arrow is the quotient map giving $C(G)$ its topology (by definition), while the bottom map is
  \begin{equation*}
    \widehat{G}\ni \rho\mapsto \chi\in\widehat{Z(G)}\text{ with }\rho|_{Z(G)}\cong \chi^{\oplus S}
  \end{equation*}  
  (for some index set $S$). This map too makes $\widehat{Z(G)}$ into a quotient topological space of $\widehat{G}$: it is continuous \cite[Proposition F.3.4]{bdv}, surjective (\Cref{le:surj}), and closed (\Cref{le:iscl}); it must thus be a quotient map by \cite[\S 22, discussion preceding Example 1]{munk}. To conclude, observe that the bijectivity of $\cat{can}$ identifies the two quotient topologies.
\end{proof}

As a consequence, the central requirement for cc-duality is the {\it in}jectivity of \Cref{eq:can}.

\begin{corollary}\label{cor:injenough}
  A locally compact group $G$ is cc-dual if an only if the canonical map \Cref{eq:can} is one-to-one.
\end{corollary}
\begin{proof}
  Apply \Cref{le:surj} and \Cref{le:bijenough} jointly.
\end{proof}

\begin{lemma}\label{le:iscl}
  Let $A$ be a $C^*$-algebra and $Z\subseteq A$ its center. The map $\widehat{A}\to \widehat{Z}$ defined by
  \begin{equation*}
    \widehat{A}\ni \rho\mapsto \chi\in \widehat{Z}\text{ such that }\rho|_Z\cong \chi^{\oplus S}
  \end{equation*}
  is closed with respect to the Fell topologies.
\end{lemma}
\begin{proof}
  It follows from \cite[\S 3.2.2]{dixc} that the closed subsets of $\widehat{A}$ are in bijection with the closed two-sided ideals, via
  \begin{equation*}
    \text{ideal }I\mapsto \widehat{A/I}. 
  \end{equation*}
  Consider such a closed subset $\widehat{A/I}\subseteq \widehat{A}$, consisting of precisely those irreducible representations whose kernels contain an ideal $I$. The kernels of the restrictions
  \begin{equation*}
    \rho|_Z,\ \rho\in \widehat{A/I}
  \end{equation*}
  contain $Z\cap I$, and on the other hand, because $Z/Z\cap I$ embeds into $A/I$, {\it every} character of $Z/Z\cap I$ can be obtained in this manner \cite[Proposition 2.10.2]{dixc}. In short, the image of the arbitrary closed subset $\widehat{A/I}\subseteq \widehat{A}$ through the map $\widehat{A}\to \widehat{Z}$ in the statement is nothing but 
  \begin{equation*}
    \widehat{Z/Z\cap I}\subseteq \widehat{Z}. 
  \end{equation*}
  This is a closed set, finishing the proof.
\end{proof}

According to the already-cited \cite[\S 22, discussion preceding Example 1]{munk}, {\it open} (as opposed to closed) continuous surjections are also quotient maps. Unlike closure though, openness is not automatic for the maps $\widehat{A}\to \widehat{Z}$ discussed in \Cref{le:iscl} (even for groups).

\begin{example}\label{ex:notop}
  For an example of a locally compact group $G$ with $\widehat{G}\to\widehat{Z(G)}$ not open it will suffice to consider $G$
  \begin{itemize}
  \item with {\it Kazhdan's property (T)} \cite[Definition 1.1.3]{bdv}, so that the trivial representation constitutes an open point of $\widehat{G}$;
  \item and with discrete infinite center, ensuring that the trivial character is {\it not} open in $\widehat{Z(G)}$. 
  \end{itemize}
  A concrete example is the universal cover $\widetilde{Sp(4,\bR)}$ of the $4\times 4$ symplectic group: it has property (T) by \cite[Theorem 6.8]{hkl} (where $Sp(4)$ is denoted by $Sp(2)$), and its center is easily computed as $\bZ$.
\end{example}

\begin{lemma}\label{le:condiscc}
  Let $G$ be a connected locally compact group and $N\trianglelefteq G$ a closed, normal, discrete subgroup. The quotient $\pi:G\to G/N$ then induces a surjection between the centers of the two groups.
\end{lemma}
\begin{proof}
  That $\pi$ maps center to center follows from surjectivity, so the substance of the claim is that the restriction
  \begin{equation*}
    \pi|_{Z:=Z(G)}:Z\to Z(G/N)
  \end{equation*}
  is again onto. 

  Connected locally compact groups are the inverse limit of their Lie-group quotients by compact normal subgroups contained in arbitrarily small open neighborhoods of $1$ \cite[\S 4.6, Theorem]{mz}. In particular, they are all {\it pro-Lie groups} in the sense of \cite[pp.161-162, equivalent definitions A, B and C]{hm-pro}.
  
  Because $G$ and $G/N$ are connected pro-Lie groups their Lie algebras (\cite[Definition 2.6]{hm-pro})
  \begin{equation*}
    Lie(G) := \mathrm{Hom}((\bR,+),G)\text{ with the compact-open topology}
  \end{equation*}
  (and similarly for $G/N$) generate dense subgroups via the exponential maps \cite[Corollary 4.22 (i)]{hm-pro}. It follows that an element of $G$ or $G/N$ is central if and only if it acts trivially on the respective Lie algebra via the adjoint action (\cite[Definition 2.27 and subsequent discussion]{hm-pro}), supplying the two outer equivalences in \Cref{eq:gpig} below.
 
  Because $\pi:G\to G/N$ has discrete kernel it identifies the two Lie algebras, securing the middle equivalence; all in all, for $g\in G$ we have
  \begin{align*}
    \pi(g)\text{ is central}&\iff Ad_{\pi(g)}=\id\text{ on }Lie(G/N)\\
                            &\iff Ad_{g}=\id\text{ on }Lie(G)\numberthis \label{eq:gpig}\\
                            &\iff g\text{ is central in }G.
  \end{align*}
  This concludes the proof.
\end{proof}

In general, neither the connectedness of $G$ nor the discreteness of $N$ can be left out in \Cref{le:condiscc}:

\begin{example}\label{ex:heis}
  Two incarnations of the Heisenberg group will illustrate both claims. First, in its real version
  \begin{equation}\label{eq:heisr}
    H=\left\{
      \begin{pmatrix}
        1&x&y\\
        0&1&z\\
        0&0&1
      \end{pmatrix}\ \bigg|\ x,y,z\in \bR
    \right\}
  \end{equation}
  it shows that the discreteness of $N$ is necessary, as we can take $G=H$ and
  \begin{equation}\label{eq:ncent}
    N:=Z(G) = \left\{
    \begin{pmatrix}
        1&0&y\\
        0&1&0\\
        0&0&1
      \end{pmatrix}
    \right\}.
  \end{equation}
  The quotient $G/N$ is then (isomorphic to) the abelian group $(\bR^2,+)$, so $\pi$ does {\it not} induce a surjection on centers.
  
  On the other hand, consider the discrete version of $H$, defined as \Cref{eq:heisr} but with entries ranging only over the integers. Taking $G=H$ and the center \Cref{eq:ncent} for $N$ again we once more have an abelian quotient $G/N$.

  In a variation of this last example we can even take $G$ finite, by allowing the elements $x$, $y$ and $z$ of \Cref{eq:heisr} to range only over $\bZ/n$ for some $n$.
\end{example}

A number of the preceding results now allow us to lift up cc-duality from quotients.

\begin{corollary}\label{cor:quotenough}
  Let $G$ be a connected locally compact group and $N\trianglelefteq G$ a closed, normal, discrete subgroup. If $G/N$ is cc-dual then so is $G$.
\end{corollary}
\begin{proof}
  Suppose $G/N$ is cc-dual. According to \Cref{cor:injenough}, it will be enough to show that an irreducible representation $\rho\in\widehat{G}$ is trivial in the chain group $C(G)$ as soon as it is trivial on the center. Now, because $N$ is a discrete normal subgroup of a {\it connected} group, it must be central. It follows that $\rho$ is trivial on $N$ and hence
  \begin{equation*}
    \rho\in\widehat{G/N}\subseteq \widehat{G}. 
  \end{equation*}
  But then, by the cc-duality assumption, $\rho$ is trivial in $C(G/N)$, which in turn maps to $C(G)$ via the contravariant functoriality noted in \Cref{le:fnct}. This finishes the proof.
\end{proof}

It should come as no surprise that The passage from $\mathrm{Rep}(G)$ to $C(G)$ loses much information, as the cc-duality of compact groups \cite[Theorem 3.1]{mug} suggests. Another incarnation of this information-loss phenomenon is that $C(G)$ is only aware of the {\it reduced} dual of $G$; following \cite[Definition 18.3.1]{dixc} or \cite[Appendix H]{hrp}:

\begin{definition}\label{def:redd}
  For a locally compact group $G$ the {\it reduced dual} $\widehat{G}_{red}$ is the set of isomorphism classes of irreducible unitary representations weakly contained in the regular representation $L^2(G)$ with respect to a Haar measure.

  Equivalently \cite[\S 18.3.2]{dixc}, these are the irreducible unitary representations of the reduced (as opposed to full or maximal) group $C^*$-algebra $C^*_{red}(G)$ (for full/reduced group $C^*$-algebras see \cite[Definition 6]{hrp} or \cite[Definition F.4.6]{bdv}).
\end{definition}

The observation alluded to above is

\begin{proposition}\label{pr:redd}
  For any locally compact group $G$, the canonical map
  \begin{equation}\label{eq:cgred}
    \widehat{G}_{red}\subset \widehat{G}\to C(G)
  \end{equation}
  is onto.
\end{proposition}
\begin{proof}
  This follows, essentially, from the absorption property of the left regular representation $\Lambda_G$ of $G$ with respect to a left Haar measure \cite[Corollary E.2.6 (ii)]{bdv}: for any unitary representation $\rho$,
  \begin{equation}\label{eq:abs}
    \rho \otimes \lambda_G\cong \lambda_G^{\oplus \dim \rho}.
  \end{equation}
  Fix an arbitrary $\rho\in \widehat{G}$ and also an irreducible $\lambda\in \widehat{G}_{red}$ (i.e. weakly contained in $\lambda_G$). We then have
  \begin{equation}\label{eq:multprec}
    \rho\otimes \lambda \preceq \rho\otimes \lambda_G\preceq \lambda_G,
  \end{equation}
  where the first weak containment follows from the Fell continuity of tensor products \cite[Proposition F.3.2]{bdv} and the second is a consequence of \Cref{eq:abs}. In particular, any irreducible weakly contained in the left-hand side $\rho\otimes \lambda$ of \Cref{eq:multprec} will belong to $\widehat{G}_{red}$. Denote, temporarily, the image of \Cref{eq:cgred} by $C(G)_{red}$; the preceding discussion then amounts to
  \begin{equation}\label{eq:rcg}
    \rho\cdot C(G)_{red} \subseteq C(G)_{red}. 
  \end{equation}
  Next, observe that $C(G)_{red}$ is
  \begin{itemize}
  \item closed under inverses, by condition \Cref{item:3} of \Cref{def:chain}, because $\lambda_G$ is self-contragredient and hence $\widehat{G}_{red}$ is invariant under the contragredient operation;
  \item closed under products, because the tensor product descends to representations of $C^*_{red}(G)$ and hence $\widehat{G}_{red}$ is closed under taking simple constituents of (i.e. irreducible representations weakly contained in) tensor products.
  \end{itemize}
  $C(G)_{red}\subseteq C(G)$ is, in other words, a subgroup, so \Cref{eq:rcg} shows that the arbitrary $\rho\in C(G)$ is a member thereof.
\end{proof}

\subsection{Almost-connected groups and their Lie quotients}

We will see that as far as {\it connected} locally compact groups $G$ go, the study of the chain group $C(G)$ can be reduced to Lie groups. In fact, this can be extended slightly; consider the following class of groups (following, for instance, \cite[p. viii]{hm}).

\begin{definition}
  A topological group $G$ is {\it almost-connected} if its quotient $G/G_0$ by the connected identity component is compact.
\end{definition}

To continue, recall from \cite[\S 4.6, Theorem]{mz} that an almost-connected locally compact group $G$ is {\it pro-lie} in the sense of \cite[Theorem 3.39]{hm-pro}:
\begin{itemize}
\item every neighborhood $U\subset G$ of the identity
\item contains a compact normal subgroup $K\trianglelefteq G$
\item such that $G/K$ is Lie;
\item and furthermore, as a consequence, $G$ is the limit (in the category of topological spaces) of the filtered system of Lie quotients $G/K$ consisting of the surjective connecting maps
  \begin{equation*}
    G/K_1\to G/K_2
  \end{equation*}
  resulting from inclusions $K_1\subset K_2$. 
\end{itemize}
We write $\cat{lieq}(G)$ for the set of compact normal subgroups $K\trianglelefteq G$ with corresponding Lie quotients $G/K$. Dualizing this inverse-limit expression
\begin{equation*}
  G\cong \varprojlim_{K\in \cat{lieq}(G)} G/K,
\end{equation*}
we have

\begin{theorem}\label{th:speclim}
  Let $G$ be an almost-connected locally compact group. The canonical inclusions
  \begin{equation*}
    \widehat{G/K}\subseteq \widehat{G},\ K\in \cat{lieq}(G)
  \end{equation*}
  realize $\widehat{G}$ as a filtered colimit
  \begin{equation}\label{eq:injlim}
    \widehat{G}\cong \varinjlim_{K\in \cat{lieq}(G)}\widehat{G/K}
  \end{equation}
  in the category of topological spaces. 
\end{theorem}
\begin{proof}
  The goal is twofold:
  \begin{enumerate}[(a)]
  \item\label{item:7} proving \Cref{eq:injlim} set-theoretically, i.e. showing that $\widehat{G}$ is the union of its subsets $\widehat{G/K}$ as $K$ ranges over $\cat{lieq}(G)$;
  \item\label{item:8} and showing that the Fell topology of $\widehat{G}$ is generated by the underlying Fell topologies of its subspaces $\widehat{G/K}$ in the sense, say, of \cite[discussion following Proposition A.5]{hatch}: a subset $U\subset \widehat{G}$ is open if and only if all intersections $U\cap \widehat{G/K}$ are. 
  \end{enumerate}
  
  \vspace{.5cm}

  {\bf \Cref{item:7}: $\widehat{G}$ is the union of $\widehat{G/K}$.} Consider an irreducible representation $\rho\in \widehat{G}$ and fix some $K\in \cat{lieq}(G)$. By \Cref{le:isot} we have a decomposition
  \begin{equation*}
    \rho|_K\cong \bigoplus_{\pi\in S}\pi^{\oplus n_{\pi}}
  \end{equation*}
  for a finite set $S\subseteq \widehat{K}$. Since irreducible unitary representations of compact groups are finite-dimensional \cite[Theorem 15.1.3]{dixc}, $\pi$ factors through some Lie group quotient
  \begin{equation*}
    K\to \overline{K}\subseteq U(n) 
  \end{equation*}
  of $K$.

  The kernel $N$ of $K\to \overline{K}$ is normal in the compact group $K$. Now, the proof of \cite[Theorem 2]{iw} shows that the connected component $G_0$ is contained in the product $K Z_G(K)$, so any normal subgroup of $K$ will automatically be normalized by $G_0$. On the other hand, the quotient
  \begin{equation*}
    G/G_0\to G/G_0K\cong (G_0K/G_0)/(G/G_0)
  \end{equation*}
  is both profinite (as a quotient of a profinite group \cite[Chapter 1, Exercise E1.13]{hm}) and Lie, so it must be finite. This means that $K$ maps to a finite-index subgroup of the profinite group $G/G_0$; if $s_i\in G$, $1\le i\le n$ is a system of representatives for the cosets then
  \begin{equation*}
    N_{small}:=\bigcap_{i=1}^n s_i N s_i^{-1}\trianglelefteq K
  \end{equation*}
  is a normal subgroup
  \begin{itemize}
  \item whose corresponding quotient $K/N_{small}$ is Lie \cite[Lemma 4.7.1]{mz};
  \item which is normalized by the connected component $G_0$, as noted above;
  \item whose image in $G/G_0$ is normal, by construction;
  \item and hence normal in $G$.
  \end{itemize}
  We now have an extension
  \begin{equation*}
    1\to K/N_{small}\to G/N_{small}\to G/K\to 1
  \end{equation*}
  where the outer terms Lie, meaning that the middle term $G/N_{small}$ must be Lie as well \cite[Theorem 7]{iw}. To summarize, $N_{small}$ is a member of $\cat{lieq}(G)$. Since the arbitrary $\rho\in\widehat{G}$ is by construction trivial on $N_{small}$, we are done with \Cref{item:7}.

  \vspace{.5cm}

  {\bf \Cref{item:8}: The Fell topology on $\widehat{G}$ is the colimit of the Fell topologies of $\widehat{G/N}$.} We have already established that we have a union
  \begin{equation*}
    \widehat{G} = \bigcup_{K\in \cat{lieq}(G)}\widehat{G/K}. 
  \end{equation*}
  It now remains to observe that the $\widehat{G/K}$ are all {\it open} in $\widehat{G}$ (in addition to being closed). To see this, recall first \cite[Proposition F.3.4]{bdv} that restricting $G$- to $K$-representations is a Fell-continuous operation; we also saw in the course of the proof of \Cref{le:isot} that \cite[Proposition 4.68]{kt} applies, and hence every restriction
  \begin{equation*}
    \rho|_K,\ \rho\in \widehat{G}
  \end{equation*}
  lives on a (finite) $G$-orbit of the discrete set $\widehat{K}$. But we then have a continuous map
  \begin{equation*}
    \widehat{G}\ni \rho\mapsto \left(\text{orbit where $\rho|_K$ lives}\right)\in \widehat{K}/G,
  \end{equation*}
  to a discrete set, and $\widehat{G/K}$ is nothing but the preimage of the open singleton $\{\text{trivial orbit}\}$. We conclude via the general (and trivial) remark that a filtered union of open subspaces is automatically a colimit topology.
\end{proof}

\begin{remark}
  Under certain circumstances, a ``transitivity of normality'' obtains for topological groups: according to \cite[Theorem 3]{iw},
  \begin{equation*}
    \text{(closed normal) in (compact normal) in (connected)}\Rightarrow \text{normal}.
  \end{equation*}
  The construction of the globally-normal subgroup $N_{small}\trianglelefteq G$ in the proof of \Cref{th:speclim} runs along the same conceptual lines; in fact, were $G$ connected (rather than {\it almost-}connected), we could have simply skipped constructing $N_{small}$ altogether and worked directly with the (already normal, by \cite[Theorem 3]{iw}) group $N$.
\end{remark}

\begin{lemma}\label{le:isot}
  Let $G$ be an almost-connected locally compact group and $K\trianglelefteq G$ a compact normal subgroup.

  For an irreducible representation $\rho\in \widehat{G}$ the restriction $\rho|_K$ is a sum
  \begin{equation*}
    \rho|_K\cong \bigoplus_{\pi\in S}\pi^{\oplus n_{\pi}}
  \end{equation*}
  for a finite subset $S\subseteq \widehat{K}$ and multiplicities $n_{\pi}$. Furthermore, $S$ is a singleton if $G$ is connected. 
\end{lemma}
\begin{proof}
  In the language of \cite[Definition 4.59]{kt}, we have to argue that the restriction $\rho|_K$ {\it lives on} a finite subset $S\subseteq \widehat{K}$, and that that set is a singleton if $G$ is connected.
  
  The conjugation action by $G$ on $K$ induces one on the discrete space $\widehat{K}$, as in \cite[\S 4.8]{kt} (the focus there is on the space $\mathrm{Prim}(K)$ of primitive ideals of $C^*(K)$, but this coincides with $\widehat{K}$ for compact groups). The connected component $G_0$ acts trivially (as it must on any discrete set), while the compact quotient $G/G_0$ acts with finite orbits.

  According to \cite[Proposition 4.68]{kt} the restriction $\rho|_K$ lives on a $G$-orbit of $\widehat{K}$; since we have just established that orbits are finite and singletons when $G=G_0$, we will be done once we show that the first hypothesis of \cite[Proposition 4.68]{kt} is satisfied.

  Specifically, that hypothesis requires that $K\le G$ be {\it regularly embedded} in the sense of \cite[paragraph following Lemma 4.67]{kt}: the quotient space $\widehat{K}/G$ is {\it almost Hausdorff}. Whatever the latter condition (defined in \cite[paragraph before Remark 4.19]{kt}) means, it will certainly be satisfied by the discrete space $\widehat{K}/G$; this finishes the proof.
\end{proof}

Switching back to chain groups, this implies

\begin{corollary}\label{cor:cgcolim}
  The chain group of a connected locally compact group $G$ can be realized as a filtered colimit
  \begin{equation*}
    C(G)\cong \varinjlim_{K\in \cat{lieq}(G)}C(G/K)
  \end{equation*}
  in the category of groups-with-topology via the contravariant functoriality provided by \Cref{le:fnct}.
\end{corollary}
\begin{proof}
  This is a consequence of \Cref{th:speclim}, since all constructions involved in passing from $\widehat{G}$ to $C(G)$ (the generators-and-relations definition of the group structure and the quotient topology) can be cast as colimits and hence commute with $\varinjlim_K$ \cite[Chapter IX, ending paragraph of \S 8]{mcl}.
\end{proof}

The analogous claim holds on the right-hand, center side of \Cref{def:ccdual}.

\begin{proposition}\label{pr:centlim}
  For a connected locally compact group $G$ the duals of the canonical maps $Z(G)\to Z(G/K)$ express the dual center as the filtered colimit
  \begin{equation*}
    \widehat{Z(G)}\cong \varinjlim_{K\in \cat{lieq}(G)} \widehat{Z(G/K)}
  \end{equation*}
  in the category $\mathrm{TopGp}$ of topological groups.
\end{proposition}
\begin{proof}
  The center $Z(G/K)$ of each $G/K$ can be recovered as the equalizer \cite[Chapter III, \S 4]{mcl} of all of the automorphisms of $G/K$ induced by elements $g\in G$. It follows, from the general principle that limits commute with limits (\cite[Chapter IX, \S 2, dual to (2)]{mcl}), that
  \begin{equation*}
    Z(G)\cong \varprojlim_{K} Z(G/K)
  \end{equation*}
  as topological groups. The conclusion follows by taking duals and using the fact that duality $\bullet \mapsto \widehat{\bullet}$ is a contravariant, involutive autoequivalence of the category of locally compact abelian groups \cite[Proposition 7.11, Theorem 7.63]{hm}.
\end{proof}

Aggregating a number of these results, we have

\begin{corollary}\label{cor:lieenough}
  A connected locally compact group $G$ is cc-dual in the sense of \Cref{def:ccdual} if all of its Lie quotients $G/K$ by compact normal subgroups are.
\end{corollary}
\begin{proof}
  This is a direct consequence of \Cref{le:fnct}, \Cref{cor:cgcolim} and \Cref{pr:centlim}.
\end{proof}

\subsection{Reciprocity for type-I locally compact groups}\label{subse:rec}

\Cref{pr:rec} is not used directly in the sequel, but it seems pertinent to understanding and working with the chain group.

Let $G$ be a type-I, second-countable locally compact group. As explained in \cite[Theorem 14.11.2]{wal2} (or \cite[Theorem 18.8.1]{dixc}), we have a well-defined (up to equivalence) Borel measure $\mu_G$ on $\widehat{G}$ (the {\it Plancherel measure} of $G$) such that the regular representation $L^2(G,\lambda)$ with respect to a left-invariant Haar measure decomposes as a direct integral \cite[\S 14.9]{wal2}
\begin{equation*}
  L^2(G,\lambda)\cong \int_{\widehat{G}} \rho^{\oplus\dim\rho}\ \mathrm{d}\mu_G(\rho). 
\end{equation*}
By \cite[Proposition 8.6.8]{dixc}, those $\rho$ contained in the {\it support} $\mathrm{supp}(\mu_G)$ (i.e. the smallest closed subset of $\widehat{G}$ whose complement has $\mu_G$-measure zero) are precisely the elements of the reduced dual (\Cref{def:redd}). 

The main reciprocity result of \cite{mr-rp} can be recast as follows (`$*$' superscripts denoting contragredient representations, as before).

\begin{proposition}\label{pr:rec}
  Let $G$ be a type-I second-countable locally compact group and
  \begin{equation*}
    \pi_i\in \widehat{G}_{red},\ 1\le i\le 3. 
  \end{equation*}
  If $\pi_3^*\preceq \pi_1\otimes \pi_2$ then
  \begin{equation*}
    \pi_i^*\preceq \pi_j\otimes \pi_k,\ \text{for all permutations }(i,j,k)\text{ of }(1,2,3). 
  \end{equation*}
\end{proposition}
\begin{proof}
  Consider the direct-integral decomposition
  \begin{equation}\label{eq:pisdec}
    \pi_1\otimes \pi_2\cong \int_{\widehat{G}} (\pi_3^*)^{\oplus n(\pi_1,\pi_2,\pi_3)}\ \mathrm{d}\mu_{\pi_1,\pi_2}(\pi_3)
  \end{equation}
  of \cite[p.362]{mr-rp}, with the slight caveat that our $\mu_{\pi_1,\pi_2}$ is that source's $\overline{\mu}_{\pi_1,\pi_2}$. Here $n(\cdot,\cdot,\cdot)$ is the relevant multiplicity function ({\it positive}-valued; see \Cref{re:pos}). According to \cite[Proposition 8.6.8]{dixc} we have $\pi_3^*\preceq \pi_1\otimes \pi_2$ precisely when $\pi_3$ is contained in the support of $\mu_{\pi_1,\pi_2}$:
  \begin{equation}\label{eq:ularge}
    \mu_{\pi_1,\pi_2}(U)>0\text{ for all open neighborhoods }\pi_3\in U\subseteq \widehat{G}. 
  \end{equation}
  Next note that this implies
  \begin{equation*}
    \mu_{\pi'_1,\pi'_2}(U)>0\text{ for }(\pi'_1,\pi'_2)\text{ in a neighborhood of }(\pi_1,\pi_2). 
  \end{equation*}
  Indeed, otherwise we could find Fell-convergent nets
  \begin{equation*}
    \pi_{i,\alpha}\to \pi_i,\ i=1,2
  \end{equation*}
  such that all $\mu_{\pi_{1,\alpha},\pi_{2,\alpha}}$ are carried by the closed set $F:=\widehat{G}_{red}\setminus U$ and hence all $\pi_{1,\alpha}\otimes \pi_{2,\alpha}$ are representations of the quotient of $C^*(G)_{red}$ corresponding to that closed subset. By the Fell-continuity of the tensor product \cite[Proposition F.3.2]{bdv}, this contradicts \Cref{eq:ularge}.

  What all of this allows us to conclude is that $\pi_3^*\preceq \pi_1\otimes \pi_2$ precisely when
  \begin{equation}\label{eq:3meas}
    (\pi_1,\pi_2,\pi_3)\in\mathrm{supp}\left(\mu_{\pi_1,\pi_2}\ \mathrm{d}\mu_G(\pi_1)\ \mathrm{d}\mu_G(\pi_2)\right),
  \end{equation}
  where $\mu_G$ is, as before, the Plancherel measure of $G$. Now, the (class of the) measure on the right-hand side of \Cref{eq:3meas} is proven in \cite[\S 2, Theorem]{mr-rp} (where that measure is denoted by $\lambda$) to be invariant under permuting the $\pi_i$, hence the conclusion.
\end{proof}

\begin{remark}\label{re:pos}
  Although this does not seem to be spelled out in \cite{mr-rp}, the multiplicity function $n(\cdot,\cdot,\cdot)$ appearing there and in the proof of \Cref{pr:rec} is positive rather than non-negative. That is, the failure of a representation $\pi_3$ to appear in the decomposition \Cref{eq:pisdec} is encoded in the vanishing of $\mu_{\pi_1,\pi_2}$ rather than that of $n(\pi_1,\pi_2,\pi_3)$.

  This is clear from the claim made in \cite[\S 2, following the first equation display]{mr-rp} that $\mu_{\pi_1,\pi_2}$ is uniquely defined up to equivalence: if we were to allow vanishing $n(\cdot,\cdot,\cdot)$ then $\mu_{\pi_1,\pi_2}$ would be arbitrary over those $\pi_3$ where $n(\pi_1,\pi_2,\pi_3)$ does vanish.
\end{remark}

\begin{remark}
  Because the measure on the right-hand side of \Cref{eq:3meas} (i.e. the $\lambda$ of \cite[\S 2]{mr-rp}) is assembled by integrating $\mu_{\pi_1,\pi_2}$ against the Cartesian square of the Plancherel measure $\mu_G$ supported on $\widehat{G}_{red}$, \Cref{pr:rec} is only valid for members of the {\it reduced} dual.

  There is thus no contradiction between \Cref{pr:rec} and the fact that even though we always have $\pi\preceq 1\otimes \pi$, we rarely have $1\preceq \pi\otimes \pi^*$ (with that containment closely linked to amenability; cf. \Cref{re:o-o}).
\end{remark}

\section{Chain-center duality for various classes of groups}\label{se:classes}

\subsection{Compact-by-abelian groups}\label{subse:cpct-by-ab}

We extend the previously-mentioned results minimally, so as to house both \Cref{ex:cpct} and \Cref{le:ab} under a common umbrella.

\begin{definition}\label{def:cpct-by-ab}
  A locally compact group $G$ is an {\it extension of a compact group by an abelian group}, or {\it compact-by-abelian} for short, if it fits into an exact sequence
  \begin{equation}\label{eq:exseq}
    1\to N\to G\to K\to 1
  \end{equation}
  with $N\le G$ closed and abelian and $K\cong G/N$ compact.
\end{definition}

\begin{theorem}\label{th:cpct-by-ab}
  Compact-by-abelian locally compact groups are cc-dual. 
\end{theorem}

This will require some background, including analyzing the representations of $G$ in terms of those of its closed normal (in this case abelian) subgroup $N\trianglelefteq G$. The general theory dealing with this is sometimes referred to as the ``Mackey machine'' for G.W. Mackey (e.g. \cite[\S 3.8]{mack-bk}); we follow the treatment in \cite[Chapter 4, especially \S 4.3]{kt} (where the phrase is ``Mackey analysis'' instead).

Various technical conditions have to be met in order for the theory to go through in full force, but they are in our case by \cite[Remark 4.26 (3)]{kt} because the quotient $K:=G/N$ is assumed compact. For this reason, the various results from \cite[Section 4.3]{kt} cited below hold in the present setup. The main such result is \cite[Theorem 4.27]{kt}, which we can summarize as follows:

\begin{itemize}
\item Every irreducible representation of $G$, when restricted to $N$, lives on (\cite[Definition 4.59]{kt}) an orbit $O$ of the conjugation action `$\triangleright$' of $G$ on $\widehat{N}$, defined as usual by
  \begin{equation}\label{eq:conjact}
    g\triangleright\chi:=\chi(g^{-1}\cdot g)\quad\text{for}\quad g\in G,\ \chi\in \widehat{N}.
  \end{equation}
\item For a character $\chi\in \widehat{N}$ with isotropy group
  \begin{equation*}
    G_{\chi} := \{g\in G\ |\ g\triangleright \chi=\chi\}
  \end{equation*}
  with respect to the conjugation action the induction
  \begin{equation*}
    \mathrm{Ind}_{G_{\chi}}^G:\mathrm{Rep}(G_{\chi})\to \mathrm{Rep}(G)
  \end{equation*}
  gives a bijection between the irreducible representations of $G_{\chi}$ whose restrictions to $N$ live on $\{\chi\}$ and those of $G$ which live on the orbit $O:=G\triangleright \chi$. 
\end{itemize}

In first instance, we can essentially reduce \Cref{th:cpct-by-ab} to the case when the abelian normal subgroup $N\trianglelefteq G$ is in fact {\it central}.

\begin{lemma}\label{le:zgn}
  Let $G$ be a compact-by-abelian group fitting into an exact sequence \Cref{eq:exseq} and $Z:=Z_G(N)$ the centralizer of $N$ in $G$.

  The irreducible representations of $G$ that restrict trivially to $Z$ are trivial in the chain group $C(G)$.
\end{lemma}
\begin{proof}
  We will work with representations that are trivial on $N$ (as all of those trivial on $Z$ must be, since $N\le Z$), identified in the obvious manner with representations of the quotient $K:=G/N$.
 
  Consider a character $\chi\in \widehat{N}$, its isotropy group $K_{\chi}\le K$, and a representation $\rho\in\widehat{K}$ which contains $K_{\chi}$-invariant vectors; or in other words, one for which the restriction
  \begin{equation*}
    \rho|_{K_{\chi}}\in \mathrm{Rep}(K_{\chi})
  \end{equation*}
  contains the trivial representation (of $K_{\chi}$). Let also
  \begin{equation*}
    \mathrm{Ind}_{G_{\chi}}^{G}\pi\in \widehat{G}
  \end{equation*}
  be a representation whose restriction to $N$ lives on $G\triangleright \chi$ (so that it is indeed induced from a $G_{\chi}$-representation whose restriction to $N$ lives on $\{\chi\}$, by the discussion preceding the present result). Then, by the induction-restriction formula
  \begin{equation*}
    \bullet\otimes \mathrm{Ind}_{A}^{B}(\square) \cong \mathrm{Ind}_A^B(\bullet|_A\otimes \square)
  \end{equation*}
  (\cite[Theorem 2.58]{kt}), we have
  \begin{equation}\label{eq:indres}
    \rho\otimes \mathrm{Ind}_{G_{\chi}}^G\pi\cong \mathrm{Ind}_{G_{\chi}}^G(\rho|_{G_{\chi}}\otimes \pi).
  \end{equation}
  Since $\rho|_{G_{\chi}}$ contains the trivial representation by assumption, \Cref{eq:indres} will contain $\mathrm{Ind}_{G_{\chi}}^G(\pi)$ as a summand. In the chain group, then, we have
  \begin{equation*}
    \rho\cdot \mathrm{Ind}_{G_{\chi}}^G\pi = \mathrm{Ind}_{G_{\chi}}^G\pi
  \end{equation*}
  and $\rho$ is trivial.

  The $\rho\in\widehat{K}\subseteq \widehat{G}$ as above, containing $G_{\chi}$-invariant vectors for some $\chi$, generate a full, summand-closed rigid monoidal subcategory $\mathrm{Rep}_{0,N}(K)$ of $\mathrm{Rep}(K)$ all of whose members are trivial in the chain group $C(G)$. It remains to
  \begin{itemize}
  \item recall (\cite[(2.3)]{mug}) the Galois correspondence between normal closed subgroups of the compact group $K$ and full, summand-closed rigid monoidal subcategories of $\mathrm{Rep}(K)$ given by
    \begin{equation}\label{eq:gal}
      (H\trianglelefteq K)\mapsto (\text{representations trivial on $H$})\simeq \mathrm{Rep}(K/H)\subseteq \mathrm{Rep}(K);
    \end{equation}
  \item observe that via this correspondence the monoidal category $\mathrm{Rep}_{0,N}(K)$ can be identified with $\mathrm{Rep}(K/H)$, where $H$ is the intersection of all conjugates
    \begin{equation*}
      g K_{\chi} g^{-1},\ g\in G,\ \chi\in \widehat{N}. 
    \end{equation*}
  \end{itemize}
  That intersection is nothing but the image through $G\to K=G/N$ of the centralizer $Z=Z_G(N)$, hence the conclusion.
\end{proof}

\begin{lemma}\label{le:maxab}
  Let $G$ be a compact-by-abelian group fitting into an exact sequence \Cref{eq:exseq}, with $N$ {\it maximal} among normal abelian subgroups.

  Any $\rho\in\widehat{K}\subseteq \widehat{G}$ is trivial in the chain group $C(G)$.
\end{lemma}
\begin{proof}
  The maximality assumption has a number of consequences:
  \begin{itemize}
  \item First, $N$ contains the center $Z(G)$.
  \item Additionally, in the quotient $K:=G/N$ we have
    \begin{equation}\label{eq:zkk}
      Z(K)\cap Z_K(N) = \{1\};
    \end{equation}
    or: the center of $K$ and the centralizer of $K$-action on $N$ induced by conjugation in $G$ intersect trivially.
  \end{itemize}
  The first observation is obvious, as we can always adjoin the center of $G$ to $N$ thus enlarging the latter. As for the second remark, note that for any $k$ in the intersection \Cref{eq:zkk} the subgroup of $G$ generated by $N$ and the preimage of $k$ through $G\to K$ is abelian normal, so we again conclude by maximality.
  
  Now, on the one hand, we know from \Cref{le:zgn} that irreducible $K$-representations (also regarded as $G$-representations) trivial on $Z_K(N)$ are trivial in $C(G)$. On the other hand, \Cref{th:cpct-by-ab} for compact groups (which is nothing but \cite[Theorem 3.1]{mug}), applied to $K$, ensures that irreducible $K$-representations trivial on $Z(K)$ are trivial in $C(G)$.

  We can now conjoin these two observations: by and the Galois correspondence \Cref{eq:gal} the rigid monoidal subcategory of $\mathrm{Rep}(K)$ generated by all irreducibles that are trivial on {\it either} $Z(K)$ or $Z_K(N)$ is precisely
  \begin{equation*}
    \mathrm{Rep}(K/Z(K)\cap Z_K(N)) = \mathrm{Rep}(K),
  \end{equation*}
  where the equality appeals to \Cref{eq:zkk}. But this is precisely the sought-after conclusion: $\rho\in \widehat{K}\subset\widehat{G}$ are all trivial in $C(G)$.
\end{proof}

\pf{th:cpct-by-ab}
\begin{th:cpct-by-ab}
  We can assume without loss of generality (via an application of Zorn's lemma) that the abelian normal subgroup $N\trianglelefteq G$ is maximal among abelian normal subgroups, so that \Cref{le:maxab} applies and $\widehat{K}\subseteq\widehat{G}$ is annihilated upon passing to the chain group $C(G)$.

  As noted in the discussion preceding \Cref{le:zgn}, each irreducible representation $\rho\in \widehat{G}$, when restricted to $N$, lives on an orbit $O=G\triangleright \chi$ of the conjugation action \Cref{eq:conjact} on $\widehat{N}$. Now, if $\rho_i\in \widehat{G}$, $i=1,2$ live, respectively, on the orbits $O_i\subset\widehat{N}$, the irreducible constituents of the tensor product $\rho_1\otimes \rho_2$ live on the orbits contained in
  \begin{equation*}
    O_1+O_2\subset \widehat{N}. 
  \end{equation*}
  It follows that the quotient map $\widehat{G}\to C(G)$ factors through the quotient group
  \begin{equation*}
    \widehat{N}/\{\chi-g\triangleright\chi\ |\ g\in G,\ \chi\in \widehat{N}\}.
  \end{equation*}
  Pontryagin duality identifies that group with the dual of
  \begin{equation*}
    N^{G}:=\{n\in N\ |\ gng^{-1}=n,\ \forall g\in G\} = Z(G),
  \end{equation*}
  and we thus have the desired identification of $C(G)$ with $\widehat{Z(G)}$.
\end{th:cpct-by-ab}

\subsection{Nilpotent groups}\label{subse:nil}

\begin{theorem}\label{th:simp-con-nil}
  Connected, nilpotent locally compact groups are chain-center dual in the sense of \Cref{def:ccdual}.
\end{theorem}
\begin{proof}
  According to \Cref{cor:lieenough} it will be enough to handle the Lie-group quotients of such a group by compact normal subgroups. For that reason, we specialize the rest of the proof to {\it Lie} (rather than arbitrary, locally compact) nilpotent connected groups.
  
  Consider first {\it simply-connected} (nilpotent, connected) Lie groups. That is precisely the class of groups to which the orbit method of \cite{kir} applies most straightforwardly (see the discussion on \cite[p.xviii]{kir}). Per the summary in \cite[p.xix, User's guide]{kir},
  \begin{itemize}
  \item The unitary dual $\widehat{G}$ is homeomorphic to the space $\cO(G)$ of orbits in the dual Lie algebra $\fg^*$ under the coadjoint action (i.e. the space of {\it coadjoint orbits}).
  \item Having identified $X\in \widehat{G}$ with an orbit $\Omega_X$ and similarly for $Y$ and $Z$, the weak containment
    \begin{equation*}
      X\preceq Y\otimes Z
    \end{equation*}
    simply says that
    \begin{equation*}
      \Omega_X\subseteq \Omega_Y + \Omega_Z
    \end{equation*}
    (with the `$+$' denoting addition of subsets in the vector space $\fg^*$).
  \item The center $Z(G)$ is connected, as are all centers of connected nilpotent Lie groups \cite[Lemma 3]{mats}. In particular, a generic element of $Z(G)$ is of the form $\exp(x)$, $x\in \fz:=Lie(Z(G))$. Adopting the $2\pi i$-scaling convention of \cite[p.xix, User's guide, item 6]{kir}, the map \Cref{eq:can} can now be identified with
    \begin{equation}\label{eq:e2pi}
      \text{coadjoint orbit }\Omega_X\mapsto (\exp(x)\mapsto e^{2\pi i f(x)})\in \widehat{Z(G)}
    \end{equation}
    for $f\in \Omega_X\subset \fg^*$ (so that the evaluation $f(x)$ makes sense: $x\in \fz\subset \fg$, and $f$ is a functional thereon).
  \end{itemize}

  Now consider The general, non-simple-connected case: $G$ is a quotient $\widetilde{G}/\Gamma$, where $\widetilde{G}$ is the (simply-connected) universal cover and $\Gamma\subset Z(\widetilde{G})$ is discrete. The discussion above applies to $\widetilde{G}$, and implies that the coadjoint orbits of $\widetilde{G}$ that survive to give unitary representations of the quotient $G$ are precisely those consisting of $f\in \fg^*$ that take integral values on the free abelian group
  \begin{equation*}
    \exp^{-1}(\Gamma)\subset \fz:=Lie(Z(\widetilde{G})). 
  \end{equation*}
  We write $\fg^*_{\Gamma\to\bZ}$ for the set of such $f\in \fg^*$; to reiterate:
  \begin{equation*}
    \fg^*_{\Gamma\to \bZ}:=\{f\in \fg^*\ |\ f(\exp^{-1}(\Gamma))\in \bZ\}.
  \end{equation*}
  The map $X\mapsto \Omega_X$ induces an identification of the chain group $C(G)$ with the quotient abelian group
  \begin{equation*}
    \fg^*_{\Gamma\to \bZ}/\mathrm{span}\{x-Ad_g x\ |\ x\in \fg^*_{\Gamma\to\bZ},\ g\in G\};
  \end{equation*}
  that is, the space $\fg^*_{\Gamma\to\bZ}/Ad$ of coinvariants under the (co)adjoint action. Vector-space duality identifies this with
  \begin{equation*}
    \{f\in \fz^* = (\fg^{Ad})^*\ |\ f(\exp^{-1}(\Gamma))\subseteq \bZ\}
  \end{equation*}
  (where the superscript $\bullet^{Ad}$ denotes invariants under the adjoint action). it follows that the map
  \begin{equation*}
    \fg^*_{\Gamma\to\bZ}/Ad\ni\text{ class of }f\mapsto (\exp(x)\mapsto e^{2\pi i f(x)})\in \widehat{Z(\widetilde{G})}
  \end{equation*}
  that implements \Cref{eq:can} (cf. \Cref{eq:e2pi}) then identifies this space with the dual
  \begin{equation*}
    \widehat{Z(\widetilde{G})/\Gamma} \cong \widehat{Z(G)},
  \end{equation*}
  as desired (this last isomorphism being a consequence of \Cref{le:condiscc}).
\end{proof}

\begin{remark}\label{re:o-o}
  In the context of \Cref{th:simp-con-nil} condition \Cref{item:3} of \Cref{def:chain} is automatic: the trivial representation is weakly contained in $X\otimes X^*$ for any irreducible $X\in \widehat{G}$, since in coadjoint-orbit language this reads as the trivial observation that
  \begin{equation*}
    \{0\}\subseteq \Omega + (-\Omega)
  \end{equation*}
  for every $Ad$-orbit $\Omega\subset\fg^*$.

  In fact, the property that the trivial representation be weakly contained in $X\otimes X^*$ for {\it every} $X\in \widehat{G}$ characterizes the amenable locally compact groups: this follows, for instance, from \cite[Theorem 5.1 and Corollary 5.5]{bek}. Amenability is not the only way to render condition \Cref{item:3} of \Cref{def:chain} redundant though: see \Cref{re:notam} below.
\end{remark}

\subsection{Discrete groups with infinite conjugacy classes}\label{subse:icc}

Before stating the next result, recall (e.g. \cite[Chapter V, Definition 7.10]{tak1}) that a group is {\it icc} (for {\it infinite conjugacy class}) if its only finite conjugacy class is $\{1\}$. 

\begin{theorem}\label{th:icc}
  Countable discrete icc groups are cc-dual in the sense of \Cref{def:ccdual}.
\end{theorem}
\begin{proof}
  Let $G$ be such a group. The center is trivial by the icc condition, so the goal is to show that $C(G)$ is also trivial.

  According to \cite[Proposition 21]{hrp} $G$ has an irreducible representation $\pi\in \widehat{G}$ weakly equivalent (denoted below by `$\sim$') to the regular representation $\lambda_G$. The Fell continuity of tensoring together with the absorption property \Cref{eq:abs} then shows that for arbitrary $\rho\in\widehat{G}$ we have
  \begin{equation*}
    \rho\otimes\pi \sim \rho\otimes\lambda_G\sim \lambda_G\sim \pi.
  \end{equation*}
  In the group $C(G)$ this reads $\rho\cdot \pi = \pi$, which is only possible if the arbitrary $\rho\in C(G)$ is trivial.
\end{proof}

\subsection{Some easily accessible semisimple examples}\label{subse:ss-easy}

The present section, like \Cref{se:ss}, is concerned with connected semisimple Lie groups. These are all second-countable and type-I \cite[Theorem 14.6.10]{wal2}. Even better, in fact; that result says that the universal $C^*$-algebra $C^*(G)$ is {\it CCR} \cite[\S 14.6.9]{wal2}, or {\it liminal} in the sense of \cite[Definition 4.2.1]{dixc}: all operators
\begin{equation*}
  \pi(a),\ a\in C^*(G)\text{ for an irreducible }\pi:C^*(G)\to B(H)
\end{equation*}
are compact.

Although \Cref{th:sln} and \Cref{th:cplx} will be superseded by \Cref{se:ss} below, the proofs are simpler because in the cases treated here more is known about tensor products of irreducible unitary representations.

\begin{theorem}\label{th:sln}
  All connected Lie groups locally isomorphic to $SL(n,\bR)$, $n\ge 2$ are cc-dual.
\end{theorem}
\begin{proof}
  Fix a positive integer $n\ge 2$ throughout. All Lie groups as in the statement cover $PSL(n,\bR)$ with discrete kernel, so the result follows from \Cref{cor:quotenough} if we prove it for projective special linear groups. As the literature typically treats the representation theory of $SL$ rather than $PSL$ we take a detour, discussing the former in order to reach the latter. 
  
  The {\it principal-series} (e.g. \cite[Chapter VII, \S 3]{knp-rep}) representations $\pi_{\sigma,\chi}$ of $G:=SL(n,\bR)$ are those of the form
  \begin{equation*}
    \mathrm{Ind}_{MAN}^G (\sigma\otimes \chi\otimes\mathrm{triv})
  \end{equation*}
  where
  \begin{itemize}
  \item $M\subset G=SL(n,\bR)$ is the group of diagonal matrices with entries $\pm 1$ and $\sigma$ is a unitary irreducible representation thereof;
  \item $A\subset G$ is the group of positive-entry diagonal matrices and $\chi$ one of its unitary characters;
  \item $\mathrm{triv}$ is the trivial representation on the subgroup $N\subset G$ of upper triangular matrices with identity diagonal. 
  \end{itemize}
  ``Most'' $\pi_{\sigma,\chi}$ are irreducible, in the sense that this happens for $\chi$ ranging over an open dense set (\cite[Corollary 12.5.4]{wal2}); for that reason, tensor products of principal-series representations are relevant to computing $C(G)$.

  \vspace{.5cm} 
  
  {\bf Case 1: $n\ge 3$.} The direct-integral decompositions of tensor products $\pi_{\sigma,\chi}\otimes \pi_{\sigma',\chi'}$ are described in \cite[p.210, Theorem]{mart}:
  \begin{itemize}
  \item When $n$ is odd the center $Z(G)$ is trivial and $SL(n,\bR)=PSL(m,\bR)$. In that case the tensor product of {\it any} two principal-series representations is nothing but the regular representation. All $\rho\in \widehat{G}_{red}$ are thus identified in the chain group $C(G)$, so the latter must be trivial by \Cref{pr:redd}.
  \item For even $n$, when the center $Z(G)$ is $\pm 1$, the tensor product in question decomposes as
    \begin{equation*}
      \pi_{\sigma,\chi}\otimes \pi_{\sigma',\chi'} \cong \int_{\widehat{G}_{\pm}}\pi^{\oplus \aleph_0}\ \mathrm{d}\mu_G(\pi)
    \end{equation*}
    where $\mu_G$ is the Plancherel measure (in conformity with prior notation), and $\widehat{G}_{\pm}$ is the set of irreducible unitary representations whose central character is trivial (for `$+$') or not (for `$-$'). In passing to $PSL(n,\bR)=G/Z(G)$ only the $+$ portion survives, once again collapsing all of $\widehat{G/Z(G)}_{red}$ to a singleton in the chain group.
  \end{itemize}

  \vspace{.5cm}

  {\bf Case 2: $n=2$.} We can argue along the same lines as for $n\ge 3$, with the added complication that now the reduced dual, in addition to the principal series, also contains {\it discrete-series} representations \cite[Proposition 9.6]{knp-rep}: those that appear as summands of the regular representation $L^2(G)$ (rather than just being {\it weakly} contained therein).

  Now, {\it all} tensor products of two irreducible representations of $SL(2,\bR)$ are described in \cite[\S 4]{repka}. In that source's notation, in passing to $PSL(2,\bR)$, the surviving representations of interest are
  \begin{itemize}
  \item $\pi_{r,\mathrm{triv}}$ for $r\in i\bR_{\ge 0}$ (principal series);    
  \item and $T_{\pm n}$, $n\ge 2$ even (discrete series).
  \end{itemize}
  \cite[\S 4, (d)]{repka} shows that the tensor product between a discrete- and a principal-series representation weakly contains the entire principal series of $PSL(2,\bR)$, so the chain group is exhausted by the classes of the principal-series elements.

  Finally, to conclude, note that by \cite[\S 4, (a)]{repka} (or, say, \cite[\S VI B., Theorem]{mart}) we have a direct summand
  \begin{equation*}
    \pi_{r,\mathrm{triv}} \otimes \pi_{r',\mathrm{triv}} \ge_{\oplus} \int(\text{entire principal series of }PSL);
  \end{equation*}
  in other words, the tensor product of two principal-series representations again weakly contains all of the latter. As desired, then, the chain group is trivial.
\end{proof}

\begin{theorem}\label{th:cplx}
  Connected, complex semisimple Lie groups are cc-dual. 
\end{theorem}
\begin{proof}
  This is very similar in spirit to the proof of \Cref{th:sln}, this time using the decomposition of tensor products obtained in \cite{wil}.  
 
  For complex semisimple Lie groups the principal-series representations $\pi_{\lambda}$ are those induced from 1-dimensional unitary representations $\lambda$ of a Borel subgroup $B\subset G$ (termed {\it non-degenerate} principal series in \cite[Introduction]{wil}). By \cite[Theorem 4.5.9]{wil} we have a decomposition
  \begin{equation*}
    \pi_{\lambda}\otimes\pi_{\lambda'}\cong \int\pi_{\theta}^{\oplus\text{multiplicity}}\ \mathrm{d}\mu(\theta)
  \end{equation*}
  for a measure $\mu$ supported on precisely those $\pi_{\theta}$ whose central character is the product of the central characters of $\pi_{\lambda}$ and $\pi_{\lambda'}$. Once more, then, \Cref{eq:can} is one-to-one when restricted to the principal series. 

  To conclude, note that in the complex case the principal series comprises all of $\widehat{G}_{red}$. This follows from the Plancherel formula for $G$, originally due to Harish-Chandra \cite{hc-pl}; without expanding on such a broad topic here, the only thing to note is that the formula (e.g. \cite[equation (4.2.7)]{wil}) involves integration over only the principal series, and the conclusion that the latter constitutes all of $\widehat{G}_{red}$ then follows from the uniqueness of the Plancherel measure \cite[Theorem 14.11.2 (3)]{wal2}.
\end{proof}

\begin{remark}\label{re:notam}
  In both \Cref{th:sln} and \Cref{th:cplx} condition \Cref{item:3} of \Cref{def:chain} was redundant, despite these groups being non-amenable; cf. \Cref{re:o-o}
\end{remark}

\section{Arbitrary connected semisimple Lie groups}\label{se:ss}

This section's main result is

\begin{theorem}\label{th:rr1}
  Connected semisimple Lie groups are cc-dual.
\end{theorem}

The proof requires some background. 

Fix a connected semisimple Lie group (we will specialize further soon). Recall, first, the {\it Iwasawa decomposition} (\cite[Theorem 6.46]{knp-bi} or \cite[Chapter IX, Theorem 1.3]{helg}).
\begin{equation}\label{eq:iw}
  K\times A\times N\cong G,
\end{equation}
where
\begin{itemize}
\item $K\le G$ contains the center $Z(G)$ \cite[Theorem 6.31]{knp-bi} and is maximal compact in $G$ if that center is finite (which we henceforth assume);
\item $A\le G$ is closed, abelian and simply-connected;
\item $N\le G$ is closed, nilpotent and simply-connected;
\item and the isomorphism \Cref{eq:iw} is given my multiplication in $G$. 
\end{itemize}
Per \cite[Chapter IX, discussion preceding Theorem 1.1]{helg} (or \cite[Chapter VII, \S 6]{knp-bi}) the {\it real rank} of $G$ is $\dim A$. 

Defining
\begin{equation}\label{eq:defm}
  M:=Z_K(\fa:=Lie(A))
\end{equation}
(the centralizer of the Lie algebra $\fa$ under the adjoint action by $K$ on $\fg:=Lie(G)$) as in, say \cite[Chapter IX, \S 1]{helg}, $B:=MAN\le G$  is a {\it minimal parabolic subgroup} of $G$; see also \cite[Chapter VII, \S 7]{knp-bi}.

The {\it principal-series} \cite[Chapter VII, \S 3]{knp-rep} unitary representations of $G$ are those of the form
\begin{equation*}
  \pi_{\sigma,\chi}:=\mathrm{Ind}_{MAN}^G (\sigma\otimes \chi\otimes\mathrm{triv}),\ \sigma\in\widehat{M},\ \chi\in\widehat{A}
\end{equation*}
as in the proof of \Cref{th:sln}.

$\pi_{\sigma,\nu}$ need not be irreducible in full generality (there are reducible examples for $SL(2n,\bR)$, for instance \cite[Example (1), preceding Theorem 3]{ks1}), but for every $\sigma\in \widehat{M}$ the set of those $\chi\in\widehat{A}$ for which $\pi_{\sigma,\chi}$ {\it is} irreducible contains an open dense subset (\cite[Corollary 12.5.4]{wal2} or \cite[Theorem 5.5.2.1]{war1}). Furthermore, all $\pi_{\sigma,\chi}$ decompose as finite sums of irreducible representations \cite[Corollary 5.5.2.2]{war1}.

Given that most of the tensor-product-decomposition literature seems to focus on the principal series, the following auxiliary result will come in handy in arguing that that is more or less enough for the purpose of studying the chain group $C(G)$.

\begin{theorem}\label{th:dpp}
  Let $G$ be a connected semisimple Lie group with finite center.

  If $\rho,\pi\in\widehat{G}$ are irreducible unitary representations with $\pi$ principal-series the tensor product $\rho\otimes\pi$ weakly contains an irreducible summand of some principal-series representation.
\end{theorem}
\begin{proof}
  The usual notation is in place: $G=KAN$ and $\pi=\pi_{\sigma,\chi}$ is induced from the minimal parabolic subgroup $B:=MAN$. We thus have
  \begin{equation*}
    \rho\otimes\pi = \rho\otimes \mathrm{Ind}_{MAN}^G(\sigma\otimes\chi\otimes\mathrm{triv})
    \cong
    \mathrm{Ind}_{MAN}^G(\rho|_{MAN}\otimes (\sigma\otimes \chi\otimes\mathrm{triv}))
  \end{equation*}
  for some $\sigma\in\widehat{M}$ and $\chi\in\widehat{A}$, where the last isomorphism is the usual induction-restriction formula \cite[Theorem 2.58]{kt}.

  The restriction of $\rho|_{MAN}$ to $N$ is a unitary representation of that simply-connected (\cite[Theorem 6.46]{knp-bi}) nilpotent Lie group, so its {\it support} \cite[Definition 3.4.6]{dixc} (i.e. the set of irreducible unitary $\widehat{N}$-representations it weakly contains) can be identified, per \cite[p.xix, User's guide]{kir}, with a set of orbits in the dual 
  \begin{equation*}
    \fn^*:=Lie(N)^*
  \end{equation*}
  under the coadjoint action of $N$. Furthermore, because $\rho|_N$ is restricted from $MA$, that set of coadjoint orbits is in fact invariant under the coadjoint action of $MA$ on $\fn^*$.
  
  The (non-trivial) adjoint action of $A$ on the Lie algebra $\fn:=Lie(N)$ can be realized via diagonal matrices with positive real eigenvalues \cite[Chapter VI, Lemma 3.5]{helg}, so the same is true of the coadjoint action on the dual $\fn^*$. But this means that any $Ad(A)$-invariant subsets of $\fn^*$ contain $0\in\fn^*$ in their closure, and hence $\rho|_{MAN}$ will contain in its Fell closure some irreducible unitary representation trivial on $N$.

  Now, since irreducible unitary $MAN$-representations trivial on $N$ must be of the form
  \begin{equation*}
    \sigma'\otimes\chi'\otimes\mathrm{triv},\quad \sigma'\in\widehat{M},\ \chi'\in\widehat{A},
  \end{equation*}
  it follows from the Fell-continuity of the induction operation \cite[Theorem F.3.5]{bdv} that 
  \begin{equation*}
    \rho\otimes\pi
    \cong
    \mathrm{Ind}_{MAN}^G(\rho|_{MAN}\otimes (\sigma\otimes \chi\otimes\mathrm{triv}))
    \to
    \mathrm{Ind}_{MAN}^G((\sigma\otimes\sigma')\otimes \chi\chi'\otimes\mathrm{triv})
  \end{equation*}
  (the arrow denoting Fell convergence). Since the right-hand limit contains principal-series representations and the classes of the latter's irreducible constituents form a clopen subset of the reduced dual (as described, say, in \cite[\S 14.12.3]{wal2}), it follows that $\rho\otimes\pi$ must have weakly-contained an irreducible summand of some principal-series representation to begin with.
\end{proof}

The following observation will allow us to ignore the continuous parameter $\chi$ in $\pi_{\sigma,\chi}$ when handling principal-series representations.

\begin{proposition}\label{pr:nochi}
  Let $G$ be a connected semisimple Lie group with finite center with Iwasawa decomposition $G=KAN$ and a corresponding minimal parabolic subgroup $B=MAN$.

  The class in $C(G)$ of an irreducible summand $\rho\le \pi_{\sigma,\chi}$ only depends on $\sigma$.
\end{proposition}
\begin{proof}
  \cite[Theorem 1]{mart} says that
  \begin{equation}\label{eq:pipi}
    \pi_{\sigma,\chi}\otimes \pi_{\sigma',\chi'}\cong \mathrm{Ind}_{MA}^G \left((\sigma\otimes\sigma')\otimes \chi\chi'\right),
  \end{equation}
  which by \cite[Theorem 2]{mart} does not depend on $\chi$ and $\chi'$.

  Now, if $\rho\le \pi_{\sigma,\chi}$ is a simple summand then $\rho\otimes\pi_{\sigma',\chi'}$ will in turn be a summand of \Cref{eq:pipi}. As noted before, by \cite[Theorem 5.5.2.1]{war1} we can choose $\chi'$ and $\chi''$ so that $\pi_{\sigma',\chi'}$ and $\pi_{\sigma,\chi''}$ are irreducible, and every simple representation weakly contained in \Cref{eq:pipi} will then also be weakly contained in
  \begin{equation*}
    \pi_{\sigma,\chi''}\otimes\pi_{\sigma',\chi'}
  \end{equation*}
  and hence have class
  \begin{equation*}
    \pi_{\sigma,\chi''}\cdot \pi_{\sigma',\chi'}\in C(G). 
  \end{equation*}
  Since $\chi'$ and $\chi''$ can be chosen in terms of $\sigma'$ and $\sigma$ alone, independently of $\chi$, we are done.
\end{proof}

As a consequence, as mentioned above, as far as $C(G)$ goes we can focus on the principal series only.

\begin{corollary}\label{cor:pionly}
  For a semisimple connected group with finite center the canonical map \Cref{eq:can} restricted to classes of irreducible principal-series representations is onto. 
\end{corollary}
\begin{proof}
  \Cref{th:dpp} shows that an arbitrary $\rho\in\widehat{G}$ is, in the chain group, the product of two irreducible summands of principal series representations. The latter can then be replaced by irreducible principal-series $\pi_{\sigma,\chi}$ by \Cref{pr:nochi}.
\end{proof}

\Cref{pr:nochi} (together with the irreducibility of $\pi_{\sigma,\chi}$ for ``most'' $\chi$) will now allow us to define a map

\begin{equation}\label{eq:mtog}
  \widehat{M}\ni \sigma\mapsto \text{ class of }\pi_{\sigma,\chi}\in C(G)_{princ}. 
\end{equation}

\begin{proposition}\label{pr:mtog}
  For a Lie group $G$ as in \Cref{pr:nochi} the map \Cref{eq:mtog} descends to a surjective group morphism
  \begin{equation}\label{eq:mtogg}
    C(M)\to C(G). 
  \end{equation}
\end{proposition}
\begin{proof}
  Because of the independence on $\chi$ implicit in \Cref{eq:mtog}, we will typically mention generic characters $\chi$ that will not otherwise play much of a role.

  The surjectivity claim is clear from \Cref{cor:pionly}, which ensures that $C(G)$ consists of the (classes of) $\pi_{\sigma,\chi}$. What is at issue here is showing that \Cref{eq:mtog} does indeed factor through $C(M)$ and that that factorization is a group morphism. Since $M$ is compact the simpler definition of the chain group introduced in \cite[Proposition 2.3]{mug} applies and we are left having to prove:
  \begin{equation}\label{eq:iffi}
    \sigma''\le \sigma\otimes\sigma'\Rightarrow \pi_{\sigma'',\chi''} = \pi_{\sigma,\chi}\pi_{\sigma',\chi'}\text{ in }C(G),
  \end{equation}
  where the left-hand side of the implication means that $\sigma''\in \widehat{M}$ is a summand of the tensor product.

  Now, $\sigma''\le \sigma\otimes\sigma'$ together with \Cref{eq:pipi} imply that
  \begin{equation}\label{eq:indma}
    \pi_{\sigma,\chi}\otimes\pi_{\sigma',\chi'} \ge \mathrm{Ind}_{MA}^G(\sigma''\otimes \chi'')
  \end{equation}
  (containment as a direct summand) for some character $\sigma''$. Since
  \begin{equation}\label{eq:pisc}
    \pi_{\sigma'',\chi''} = \mathrm{Ind}_{MAN}^G(\sigma''\otimes\chi''\otimes\mathrm{triv})
  \end{equation}
  by definition and the right-hand side of \Cref{eq:indma} breaks up as
  \begin{equation}\label{eq:indind}
    \mathrm{Ind}_{MAN}^G\circ\mathrm{Ind}_{MA}^{MAN}(\sigma''\otimes\chi'')
  \end{equation}
  by induction in stages \cite[Theorem 2.47]{kt}, it will be enough to observe the following.
  \begin{itemize}
  \item $MAN$ being amenable (as the semidirect product of a compact group $M$ and a solvable group $AN$ \cite[Theorem 1.2.11]{rnd}), we have the weak containment
    \begin{equation*}
      \sigma''\otimes\chi''\otimes\mathrm{triv}
      \preceq
      \mathrm{Ind}_{MA}^{MAN}(\sigma''\otimes\chi'')
      \cong
      \mathrm{Ind}_{MA}^{MAN}(\sigma''\otimes\chi''\otimes\mathrm{triv})|_{MA}
    \end{equation*}
    by \cite[p.296, Theorem]{grnlf};
  \item whence \Cref{eq:pisc} is weakly contained in \Cref{eq:indind} by further inducing up to $G$, because induction is Fell-continuous \cite[Theorem F.3.5]{bdv}. 
  \end{itemize}
  The conclusion is that we have the weak containment relation
  \begin{equation*}
    \pi_{\sigma'',\chi''}\preceq \pi_{\sigma,\chi}\otimes\pi_{\sigma',\chi'},
  \end{equation*}
  hence the right-hand side of the desired implication \Cref{eq:iffi}.
\end{proof}

\Cref{pr:mtog} can be recast in somewhat different terms.

\begin{proposition}\label{pr:psim}
  Let $G=KAN$ be the Iwasawa decomposition of a connected semisimple Lie group with finite center and $B=MAN$ the corresponding minimal parabolic.
  \begin{enumerate}[(a)]
  \item\label{item:9} The map
    \begin{equation*}
      \Psi_M:\widehat{M}\ni\sigma\mapsto\left(\text{class of any irreducible }\pi\preceq \mathrm{Ind}_M^G \sigma\right)\in C(G)
    \end{equation*}
    is well defined and induces the epimorphism $C(M)\to C(G)$ of \Cref{pr:mtog}, denoted abusively by the same symbol $\Psi_M$.
  \item\label{item:10} Furthermore, $\Psi(\sigma_1)=\Psi(\sigma_2)$ provided $\mathrm{Ind}_M^K\sigma_i\in\widehat{K}$, $i=1,2$ are not disjoint.
  \end{enumerate}
\end{proposition}
\begin{proof}
  It follows from \Cref{pr:mtog} that $\pi_{\mathrm{triv},\chi}$ are all in the trivial class of $C(G)$, whereupon \Cref{eq:pipi} identifies the image of $\sigma\in\widehat{M}$ through \Cref{eq:mtogg} with
  \begin{equation*}
    \left(\text{class of any irreducible }\pi\preceq \mathrm{Ind}_{MA}^G \sigma\otimes\chi\right)\in C(G)
  \end{equation*}
  (for arbitrary $\chi\in\widehat{A}$). This coincides with 
  \begin{equation*}
    \left(\text{class of any irreducible }\pi\preceq \mathrm{Ind}_{M}^G \sigma\right)\in C(G)
  \end{equation*}
  because
  \begin{equation*}
    \mathrm{Ind}_M^{MA}\sigma \cong \sigma\otimes\left(\text{regular representation }L^2(A)\right)\cong \sigma\otimes \int_{\widehat{A}}\chi\ \mathrm{d}\mu_A(\chi),
  \end{equation*}
  proving part \Cref{item:9}.

  As for \Cref{item:10}, it is a consequence of \Cref{item:9}: if $\mathrm{Ind}_M^K\sigma_i$ have summands in common then so do
  \begin{equation*}
    \mathrm{Ind}_M^G\sigma_i\cong \left(\mathrm{Ind}_K^G\circ\mathrm{Ind}_M^K\right)\sigma_i,
  \end{equation*}
  (induction in stages \cite[Theorem 2.47]{kt}). 
\end{proof}

The advantage of \Cref{pr:psim} is that it has a counterpart for $K$ rather than $M$.

\begin{proposition}\label{pr:psik}
  Let $G=KAN$ be the Iwasawa decomposition of a connected semisimple Lie group with finite center and $B=MAN$ the corresponding minimal parabolic.
  \begin{enumerate}[(a)]
  \item\label{item:11} The map
    \begin{equation*}
      \Psi_K:\widehat{K}\ni\sigma\mapsto\left(\text{class of any irreducible }\pi\preceq \mathrm{Ind}_K^G \sigma\right)\in C(G)
    \end{equation*}
    is well defined and induces an epimorphism $C(K)\to C(G)$ denoted by the same symbol $\Psi_K$.
  \item\label{item:12} Furthermore, $\Psi(\sigma_1)=\Psi(\sigma_2)$ provided $\sigma_i|_{M}\in\widehat{M}$, $i=1,2$ are not disjoint.
  \end{enumerate}
\end{proposition}
\begin{proof}
  That $\psi_K:\widehat{K}\to C(G)$ is well defined follows from \Cref{pr:psim}: an arbitrary $\sigma\in \widehat{K}$ is a summand of some
  \begin{equation*}
    \mathrm{Ind}_M^K \sigma',\ \sigma'\in\widehat{M},
  \end{equation*}
  so that in turn
  \begin{equation*}
    \mathrm{Ind}_K^G\sigma\le \left(\mathrm{Ind}_K^G\circ\mathrm{Ind}_M^K\right) \sigma'\cong \mathrm{Ind}_M^G\sigma'.
  \end{equation*}
  By part \Cref{pr:psim} \Cref{item:9} all summands of the right-hand side are identified in $C(G)$, hence the conclusion (that $\Psi_K$ is well defined).

  If now both $\sigma_i\in\widehat{K}$, $i=1,2$ contain $\sigma'\in\widehat{M}$ when restricted to $M$, then by Frobenius reciprocity for compact groups \cite[(8.9)]{rob} we have
  \begin{equation*}
    \sigma_i\le \mathrm{Ind}_M^K\sigma',\ i=1,2. 
  \end{equation*}
  We can thus choose the same $\sigma'$ for both $\sigma_i$ in the argument above, proving part \Cref{item:12} of the statement.

  It thus remains to prove the portion of part \Cref{item:11} claiming that $\Psi_K$ induces a {\it group morphism} $C(K)\to C(G)$ (rather than just a map):
  \begin{equation*}
    \sigma\le \sigma'\otimes \sigma''\Rightarrow \Psi_K(\sigma)=\Psi(\sigma')\Psi(\sigma'')\in C(G). 
  \end{equation*}
  Once more using Frobenius reciprocity and the induction-restriction formula (\cite[Theorem 2.58]{kt}) we have
  \begin{equation}\label{eq:inindss}
    \sigma\le \sigma'\otimes\mathrm{Ind}_M^K(\sigma''|_M)\cong \mathrm{Ind}_M^K(\sigma'|_M\otimes \sigma''|_M)
  \end{equation}
  We know from \Cref{pr:psim} that further inducing the right-hand side of \Cref{eq:inindss} to $G$ produces a direct integral of irreducible $\rho\in\widehat{G}$, all identified with
  \begin{equation*}
    \psi_M(\text{irreducible}\le \sigma'|_M)\cdot \psi_M(\text{irreducible}\le \sigma''|_M)\in C(G).
  \end{equation*}
  As seen in the preceding discussion (in the present proof) this is precisely
  \begin{equation*}
    \Psi_K(\sigma')\Psi_K(\sigma'')\in C(G). 
  \end{equation*}
  Since at the representation level this object was obtained via $\mathrm{Ind}_K^G$ from a $K$-representation containing $\sigma\in\widehat{K}$ as a summand, it must also coincide with $\Psi_K(\sigma)$. 
\end{proof}

\pf{th:rr1}
\begin{th:rr1}
  It is enough to consider (connected, semisimple, Lie) groups $G$ with finite center, since their cc-duality entails their covers' cc-duality by \Cref{cor:quotenough}. 

  Consider the epimorphism
  \begin{equation*}
    \Psi_K:\widehat{K}\cong C(K)\to C(G)
  \end{equation*}
  of \Cref{pr:psik} (where the first isomorphism uses the cc-duality of the compact group $K$ \cite[Theorem 3.1]{mug}). According to \Cref{cor:injenough}, we have to argue that $\Psi_K$ annihilates (classes of) irreducible representations $\sigma\in\widehat{K}$ trivial on $Z(G)\subset K$.

  Consider such a $\sigma\in\widehat{K}$, and an irreducible summand
  \begin{equation*}
    \widehat{M}\ni\sigma'\le \sigma|_M
  \end{equation*}
  which must also be trivial on $Z(G)\le M$. Because $Z(K)\cap M=Z(G)$ (\Cref{le:zkm}), $\sigma'$ extends to an irreducible representation (also denoted by $\sigma'$) of $MZ(K)$, trivial on $Z(K)$. But then $\sigma'$ further extends to an irreducible $K$-representation $\overline{\sigma}$: 
  \begin{equation*}
    \overline{\sigma}\ =\text{ any irreducible summand of }\ \mathrm{Ind}_{MZ(K)}^K\sigma'
  \end{equation*}
  will do, by Frobenius reciprocity. Now, on the one hand $\overline{\sigma}$ is trivial on $Z(K)$  and hence trivial in $C(K)$ (so annihilated by $\Psi_K:C(K)\to C(G)$) by the cc-duality of $K$. On the other hand though,
  \begin{equation*}
    \Psi_K(\overline{\sigma}) = \Psi_K(\sigma)
  \end{equation*}
  by \Cref{pr:psik} \Cref{item:12}. This finishes the proof. 
\end{th:rr1}

\begin{lemma}\label{le:zkm}
  Let $G=KAN$ be the Iwasawa decomposition of a connected semisimple Lie group with finite center and $B=MAN$ the corresponding minimal parabolic. We then have
  \begin{equation*}
    Z(K)\cap M = Z(G). 
  \end{equation*}
\end{lemma}
\begin{proof}
  The global center $Z(G)$ is contained in $K$ \cite[Theorem 6.31]{knp-bi} and hence also in $M$ (given the latter's definition as the centralizer of $\fa=Lie(A)$ in $K$: \Cref{eq:defm}), so the inclusion
  \begin{equation*}
    Z(K)\cap M\supseteq Z(G)
  \end{equation*}
  is clear. To prove the opposite inclusion, consider first a {\it Cartan decomposition}
  \begin{equation*}
    Lie(G)=:\fg = \fk\oplus \fp,\ \fk:=Lie(K)
  \end{equation*}
  as in \cite[(6.23)]{knp-bi}, compatible with the Iwasawa decomposition $G=KAN$ as in \cite[Chapter VI, \S 4]{knp-bi}; in particular, the Lie algebra $\fa=Lie(A)$ is a maximal abelian subspace of $\fp$.

  The intersection of all conjugates $kMk^{-1}$ will centralize
  \begin{equation*}
    \bigcup_k Ad_k(\fa) = \fp,
  \end{equation*}
  where the last equality is \cite[Theorem 6.51]{knp-bi}. Since $Z(K)$ centralizes $\fk$,
  \begin{equation*}
    Z(K)\cap M\subseteq Z(K)\cap \bigcap_{k\in K}kMk^{-1}
  \end{equation*}
  centralizes both $\fk$ and $\fp$ and hence $\fg=\fk\oplus \fp$. In short, that intersection must be central in $G$.
\end{proof}



\addcontentsline{toc}{section}{References}

\Addresses

\end{document}